\documentclass[11pt]{amsart}
\usepackage{amsmath,amssymb,amsthm,amsxtra,graphics,enumerate} 
\usepackage{stmaryrd}
\usepackage{times}
\usepackage[margin=1.3in]{geometry}
\usepackage[linktocpage=true, hidelinks, colorlinks=false, allcolors=blue]{hyperref}
\usepackage{comment}
\usepackage{scalerel}
\usepackage{mathrsfs}

\newtheorem{thm}{Theorem}[section]
\newtheorem{claim}[thm]{Claim}
\newtheorem{cor}[thm]{Corollary}

\newtheorem{prop}[thm]{Proposition}

\theoremstyle{definition}

\theoremstyle{remark}
\newtheorem{rem}[thm]{Remark}

\newcommand\Fcal{\mathcal{F}}

\newcommand\ran{\mathrm{ran}}

\newcommand\Add{\operatorname{Add}}

\newcommand\mand{\textrm{ and }}

\newcommand\Ascr{\mathscr{A}}
\newcommand\Cscr{\mathscr{C}}
\newcommand\Fscr{\mathscr{F}}

\newcommand\Sscr{\mathscr{S}}

\newcommand\Uscr{\mathscr{U}}
\newcommand\Vscr{\mathscr{V}}

\newcommand\Rbb{\mathbb{R}}

\newcommand\Th{{}^{th}}

\renewcommand\emptyset{\varnothing}

\newcommand\ZFC{\mathrm{ZFC}}
\newcommand\MA{\mathbf{MA}}

\newcommand\lex{\mathrm{lex}}

\newcommand\Ht{\mathrm{ht}}
\newcommand\next{\mathrm{next}}

\newcommand\dom{\operatorname{dom}}

\title{Forcing axioms and the structure theory of Aronszajn trees}
\author{Justin Tatch Moore}
\author{Stevo Todorcevic}

\subjclass{03E35, 03E50, 03E57}

\keywords{Aronszajn, club-isomorphic, coherent tree, Countryman line, Martin's Axiom, Proper Forcing Axiom}

\address{Justin Tatch Moore, Department of Mathematics, Cornell University, Ithaca, New York, USA}
\address{Stevo Todorcevic, Department of Mathematics, University of Toronto, Toronto, Canada; Institut de Math\'ematiques de Jussieu, Paris, France; and Matemati\v{c}ki Institut, SANU, Belgrade, Serbia}

\thanks{The authors would like to thank the Sicilian Sidewalk Cafe for
their hospitality during the first author's visit to Toronto
in June 2026 when many of the results of this paper were obtained.
}
\begin{document}

\begin{abstract}
    We study which structural results about Aronszajn trees require the Proper Forcing
    Axiom and which require only $\MA_{\aleph_1}$.
    In particular, we show that $\MA_{\aleph_1}$ is consistent with the existence
    of $2^{\aleph_1}$ pairwise non-club-isomorphic coherent A-trees.
  \end{abstract}

\maketitle

\section{Introduction}

The structure theory of Aronszajn trees and the corresponding linear orderings is strongly dependent on
the additional axioms of set theory. For example, 
in \cite{AbrahamShelah} Abraham and Shelah proved that the Proper Forcing Axiom (PFA) implies that
any two Aronszajn trees are club isomorphic and demonstrated that the same
conclusion did not follow from the weaker and better known axiom $\MA_{\aleph_1}$.
In subsequent work further developing the structure theory of Aronszajn trees and Aronszajn lines,
both the hypotheses $\MA_{\aleph_1}$ and PFA played an important role \cite{l-maps}, \cite{linear_basis}, \cite{walks}, \cite{univ_A-line}.
The purpose of the current paper is to clarify the role that these two axioms play in some of these structure theorems and also to correct certain
errors and misstatements each author has made in the literature \cite[Theorem 1.1]{univ_A-line} \cite[Corollary 2.1.13]{walks}. These errors are related to the fact that many of these structure theorems
require only $\MA_{\aleph_1}$ and that applications of PFA go through  very similar proper posets of size $\aleph_1$.
Therefore it is natural to determine the structure theorems
for which PFA is really needed as well as to
examine, modulo $\MA_{\aleph_1}$, the relative strengths behind the statements which do require PFA.

We have already mentioned that the statement that any two Aronszajn trees are
club isomorphic is a consequence of PFA but not of $\MA_{\aleph_1}$.
However, an examination of the arguments of \cite{AbrahamShelah} shows that two
Aronszajn trees can be made club isomorphic by a ccc poset in the forcing
extension of the Baumgartner poset for adding a club by finite approximations.
This is also true about the statement that any pair of Lipschitz trees $S$ and $T$
are metrically isomorphic\footnote{i.e., the existence of level sequences
$\{s_\alpha: \alpha\in \Gamma\}\subseteq S$ and 
$\{t_\alpha: \alpha\in \Gamma\}\subseteq T$ indexed by some uncountable
$\Gamma\subseteq \omega_1$ such that
$\Delta(s_\alpha, s_\beta) > \Delta(s_\alpha, s_\gamma)$ is equivalent to
$\Delta(t_\alpha, t_\beta) > \Delta(t_\alpha, t_\gamma)$) for every triple
$\alpha, \beta$ and $\gamma$ of elements of $\Gamma$.}, a statement used in
\cite{l-maps} (together with $\MA_{\aleph_1}$) to show that the corresponding
ultrafilters $\mathcal{U}(S)$ and $\mathcal{U}(T)$ are Rudin-Kiesler equivalent. 
Thus, we have the following  statements of interest here listed decreasingly in
strengths modulo the assumption of $\MA_{\aleph_1}$:
\begin{enumerate}
\item Every pair $S$ and $T$ of Aronszajn trees are club isomorphic.
\item Every pair $S$ and $T$ of Lipschitz trees are club isomorphic.
\item Every pair $S$ and $T$ of Lipschitz trees are metrically isomorphic.
\item $\mathcal{U}(S)\equiv_{RK}\mathcal{U}(T)$
for every pair of Lipschitz trees $S$ and $T$.
\end{enumerate}
This paper examines these statements, their intended applications, and their
independence from $\MA_{\aleph_1}$.
In Section \ref{forcing_clubiso:sec}, 
we present conditions that
guarantee that two Aronszajn tress can be made club isomorphic by ccc posets.
For example, given two Aronszajn trees $S$ and $T$ which contain no Souslin
subtrees' the existence of a club in $\omega_1$ that is sufficiently fast relative
$S$ and $T$ guarantees that the trees can be made club isomorphic by a ccc poset.
In Section \ref{indep_MA:sec}, 
we give independence results from $\MA_{\aleph_1}$ by producing a model of
$\MA_{\aleph_1}$ in which the following statements are true:
\begin{enumerate}[(a)]
\item \label{many_C-trees} 
There is a family $\Sscr$ of $2^{\aleph_1}$ coherent A-trees
such that no pair of elements of $\Sscr$ have club-isomorphic subtrees;

\item \label{many_C-lines}
There is a family $\Cscr$ of $2^{\aleph_1}$ Countryman lines such that no pair
of elements of $\Cscr$ have uncountable suborders which are isomorphic or reverse
isomorphic;

\item 
There is an A-tree which is not club isomorphic to a coherent A-tree; 

\item 
There are two pruned A-trees $S$ and $T$ and a strictly increasing level
preserving surjection $f:S \to T$ such that $S$ and $T$ are not club isomorphic;

\item \label{non_clubiso_equiv}
There are two pruned A-trees $S$ and $T$ which are not club isomorphic and yet
$S \leq T$ and $T \leq S$.
Moreover $S$ can be taken to be coherent.

\end{enumerate}
The key idea is to produce a model $\MA_{\aleph_1}$ with a family $\mathcal{S}$ of cardinality
$2^{\aleph_1}$ of Lipschitz trees such that for each $S\in \mathcal{S}$ the
ultrafilter $\mathcal{U}(S)$ extends the club filter and such that
$\mathcal{U}(S)\neq \mathcal{U}(S')$ whenever $S\neq S'$ are in $\mathcal{S}$.

Section \ref{two_consequences:sec} concludes the paper
with two consequences of $\MA_{\aleph_1}$ concerning coherent A-trees.
The first states that if $S$ and $T$ are
Aronszajn trees with $S$ coherent, then $S\leq T$ and $T\leq S$ imply that $S$ is
club isomorphic to a subtree of $T$.
If moreover $T$ is also coherent then $S$
and $T$ are club isomorphic.
This shows that the consistency result
(\ref{non_clubiso_equiv}) above is sharp.
The second  consequence of $\MA_{\aleph_1}$ states:
\emph{if $S$ and $T$ are club-isomorphic A-trees and $S$ is coherent, then there
is a countable-to-one function $f:\omega_1 \to \omega_1$ such that
$\Uscr(S) = f^* \Uscr(T)$.}

\begin{rem}
    Lucas Polymeris has independently demonstrated that $\MA_{\aleph_1}$
    is consistent with (\ref{many_C-trees}) and (\ref{many_C-lines}) \cite{Polymeris}.
    The methods he employs are different than those in the present paper
    analogous to the methods developed in \cite{AbrahamShelah} for
    demonstrating that $\MA_{\aleph_1}$ is consistent
    with the existence of two A-trees which are not club-isomorphic.
\end{rem}

\section{Preliminaries}

In this section, we review some standard terminology which we will need.
See \cite{trees_lo:Todorcevic} for a general survey on set-theoretic trees and linear orders;
see \cite{set_theory:kunen} for background material on set theory and forcing.

A \emph{tree} is a partially ordered set $(T,\leq)$ such that for all
$t \in T$, $\{s \in T: s < T\}$ is well ordered by $\leq$.
The ordertype of the set of predecessors of $t$ is the \emph{height of $t$}, denoted $\Ht(t)$.
The $\xi\Th$-level of $T$, denoted $T_\xi$, is the set of all elements of $T$ of height $\xi$.
The least $\xi$ such that $T_\xi$ is empty is the \emph{height of $T$}.
If $E$ is a subset of the height of $T$, then we will write $T \restriction E$ to denote $\bigcup_{\xi \in E} T_\xi$,
equipped with the order inherited from $T$.
All trees in this paper are assumed to be \emph{Hausdorff}: if $s$ and $t$ are distinct elements of $T_\nu$ where $\nu > 0$ is a
is a limit ordinal, then $s$ and $t$ have different sets of predecessors.

If $\kappa$ is a regular cardinal,
a $\kappa$-tree is a tree of height $\kappa$ in which all levels have cardinality less than $\kappa$.
A $\kappa$-tree is \emph{pruned} if every element has extensions of arbitrarily large height below $\kappa$.
It is easily checked that if $T$ is a $\kappa$-tree and $U \subseteq T$ consists of those elements of $T$ which have
$\kappa$-many extensions, then $U$ is a pruned $\kappa$-tree.

An \emph{Aronszajn tree (A-tree)} is an $\omega_1$-tree which does not contain an uncountable chain.
In this paper we will be concerned exclusively with A-trees.
An $\omega_1$-tree is \emph{pruned} if every element of the tree has uncountably many elements above it.
A \emph{subtree} $U$ of an $\omega_1$-tree $T$ is an uncountable subset of $T$
which is downward closed (if $u \in U$ and $v \in T$ with $v \leq u$,
then $v \in U$).

Any A-tree can be represented as a subtree of $\omega^{<\omega_1}$, ordered by extension and we will generally
represent our A-trees in this way.
An A-tree $S$ is \emph{coherent} if it is a subtree of $\omega^{<\omega_1}$ and has the property that
whenever $s,s' \in S_\xi$ for some $\xi$, $\{\eta < \xi : s(\eta) \ne s'(\eta)\}$ is finite.
We will also sometimes say that $S$ is a coherent A-tree if it is isomorphic to a coherent A-tree.

If $T$ is an A-tree and $s,t \in T$, then $\Delta(s,t)$ is the height of $\max \{u \in T : (u < s) \mand (u < t)\}$ (the maximum exists since $T$ is Hausdorff).
Note that if $T$ is represented as a subtree of $\omega^{<\omega_1}$ and $s$ and $t$ are incomparable,
then $\Delta(s,t) = \min \{\delta : s(\delta) \ne t(\delta)\}$.
If $T$ is a coherent A-tree, define $\Uscr(T)$ to consist of all $U \subseteq \omega_1$ such that for some uncountable
$X \subseteq T$, $$\Delta(X) := \{\Delta(s,t) : s,t \in X \textrm{ are incomparable}\}$$
is contained in $U$.
The second author defined $\Uscr(T)$ and proved $\Uscr(T)$ is a uniform filter.
He moreover showed that if the c.c.c. is productive (in particular if $\MA_{\aleph_1}$ holds) then $\Uscr(T)$ is an ultrafilter
\cite{ICM:stevo} \cite{l-maps}.

An A-tree $T$ is \emph{Lipschitz} if whenever $f:T \to T$ is an uncountable partial level preserving map,
there is is an uncountable $X \subseteq \dom(f)$ such that for all $s \ne t \in X$,
$$\Delta(s,t) \leq \Delta(f(s),f(t)).$$
The notion of a Lipschitz tree was introduced by the second author in \cite{l-maps} where he analyzed the class extensively
in the presence of the Proper Forcing Axiom (PFA).
For now we note the following result connecting the notion to that of a coherent tree.

\begin{thm} \cite{l-maps}
    Coherent trees with no Souslin subtrees are Lipschitz.
    Lipshitz trees do not contain Souslin subtrees and assuming $\MA_{\aleph_1}$ any Lipschitz tree is coherent.
\end{thm}

If $S$ and $T$ are $\omega_1$-trees, we write $S \leq T$ to denote the assertion that there is a strictly increasing order
preserving map from $S$ to $T$.
Define $S \oplus T$ to be the disjoint union of $S$ and $T$ ordered
so that the canonical maps of $S$ and $T$ into $S \oplus T$ are order preserving and so that every element of $S$ is incomparable 
to every element of $T$.
Note that $S,T \leq S \oplus T$.
Define $S \otimes T = \{(s,t) \in S \times T : \Ht(s) = \Ht(t)\}$, equipped with the coordinatewise order inherited from
$S \times T$.
Observe that the projection maps witness $S \otimes T \leq S,T$.

We now recall the following definition from \cite{con_linear_basis2}:
a forcing $P$ \emph{preserves subtrees} of an $\omega_1$-tree $T$ if every subtree of $T$ which occurs in a generic
extension by $P$ contains a subtree in the ground model.
This concept originated from the following result, which itself was an extension of work of Silver \cite{silver}.

\begin{thm} \cite[Lemma 3.1]{baumbase} \cite[Lemma 8.13]{trees_lo:Todorcevic}
$\sigma$-closed forcings preserve subtrees of $\omega_1$-trees.    
\end{thm}
\noindent
We will see in the next section that the combination a poset preserving subtrees of A-trees and yet adding a club containing
no ground model club is useful in the context of trying to force two A-trees to be club isomorphic.

Next we recall the definition of Baumgartner's poset for adding a club by finite conditions (see \cite{forcing_clubs}).
Let $P$ consist of all pairs $p = (E_p,U_p)$ such that $E_p$ is a finite subset of $\omega_1$ and 
$U_p$ is a countable clopen subset of $\omega_1 \setminus U_p$.
Define $q \leq p$ if $E_p \subseteq E_q$ and $U_p \subseteq U_q$.
The union of the first coordinates of a generic filter $G \subseteq P$ is known as a \emph{Baumgartner club}.
The significant properties of this poset are that it is \emph{strongly proper},\footnote{Here recall that a condition $p$ in a poset $P$ is (strongly) $(M,P)$-generic for a suitable countable elementary submodel $M$ if
it forces that $G \cap M$ is $M$-generic ($V$-generic) for $P \cap M$; $P$ is (strongly) proper if whenever $M$ is a suitable countable elementary submodel for $P$ and $p \in M \cap P$, there is a $q \leq p$ which is (strongly) $(M,P)$-generic.
}
of cardinality $\aleph_1$ (and hence preserves
cardinals above $\aleph_1$), and no infinite subset of generic club is in the ground model.
We note here the following fact, which implies that the poset to add a Baumgartner club also preserves subtrees of A-trees.

\begin{prop} \label{strong_proper_pres_subtree}
Strongly proper posets preserve subtrees of $\omega_1$-trees.
\end{prop}

\begin{proof}
    Let $P$ be strongly proper and $T$ be an $\omega_1$-tree.
    Suppose $\dot S$ is a $P$-name for a subtree of $T$ and $p \in P$ and
    let $M$ be a countable elementary submodel of $H_\theta$ for $\theta$
    sufficiently large so that $P$, $T$, $\dot S$, and $p$ are in $M$.
    Set $\delta = M \cap \omega_1$ and $q \leq p$ be a strongly $(M,Q)$-generic.
    Since $q$ forces that  $\dot S_\delta \ne \emptyset$, by extending $q$ if 
    necessary, we may find an $s \in T_\delta$ such that $q$ forces
    $\check s \in \dot S$.
    Since $q$ is strongly proper, there is a $p'$ in $M \cap P$ such
    that $q \leq p' \leq p$ and whenever $p'' \leq p'$ is in $M \cap P$,
    $p''$ is compatible with $p'$.
    Let $U$ consist of all $u \in T$ such that $p'$ forces $\check u \in \dot S$.
    Since $U$ is defineable from parameters in $M$, $U \in M$.
    Since $\dot S$ is forced to be downward-closed, $U$ is downward closed.
    Since clearly $p$ forces $\check U \subseteq \dot S$, it suffices
    to show that $U$ is uncountable.
    By elementarity of $M$, it suffices to show that every strict predecessor
    of $s$ is in $U$.
    Let $u < s$ and note that $q$ forces $\check u \in \dot S$.
    Suppose that $p'$ does not force $\check u \in \dot S$.
    By elementarity of $M$, there is a $p'' \leq p'$ in $P \cap M$ such that
    $p''$ forces $\check u \not \in \dot S$.
    But this is impossible since $p''$ is compatible with $q$.
\end{proof}

\begin{thm} \cite{con_linear_basis2}
    Let $T$ be an $\omega_1$-tree.
    A countable support iteration of proper forcings which preserve subtrees of $T$ preserves subtrees of $T$.
\end{thm}

Finally, we recall a proposition which makes explicit the connection of isomorphisms between lexicographically ordered subsets of A-trees and club-isomorphisms between the subtrees they generate.
Such an implication seems at least implicit in \cite{AbrahamShelah}.
A special case was stated in \cite[Lemma 2.9]{minimal_unctbl} where the proof was left to the reader.
To our knowledge the first proof appears in \cite{Peng_thesis}
it was proved in the generality of $\kappa$-A-lines and $\kappa$-A-trees in \cite[Lemma 5.2]{Shalev}.

\begin{prop} \label{lo_to_club-iso} 
    Suppose that $T_0,T_1 \subseteq \omega^{<\omega_1}$ are A-trees and for each $i=0,1$,
    $X_i \subseteq T_i$ is uncountable.
    If $(X_0,\leq_\lex)$ is isomorphic to $(X_1,\leq_\lex)$, then 
    then $S_0$ and $S_1$ are club-isomorphic, where $S_i$ consists of all $s \in T_i$ having uncountably many
    extensions in $X_i$.
\end{prop}

\section{Forcing A-trees to be club isomorphic with a c.c.c. poset}
\label{forcing_clubiso:sec}

In this section we recall a standard poset to make two make two A-trees club-isomorphic and establish a criteria
for when it is ccc.
If $T$ is an A-tree and $E \subseteq \omega_1$ is closed, define
$$\Delta_E(s,t):= \sup\{ \delta \in E \cup \{0\}: \delta \leq \min(\Ht(s),\Ht(t)) \mand s \restriction \delta = t \restriction \delta\}$$
whenever $s,t \in T$.
Notice that if $\Ht(s),\Ht(t) \in E$, then $\Delta_E(s,s) = \Ht(s)$ and
$\Delta_E(s,t) = \Ht(s)$ if and only if $s \leq t$.
Thus any (possibly partial) function $f$ from $S \restriction E$ to $T \restriction E$ which preserves $\Delta_E$---i.e. 
$\Delta_E(f(s),f(t)) = \Delta_E(s,t)$ for all $s,t \in \dom(f)$---
necessarily preserves order and height.
Also observe that while $\Delta_E$ depends on $T$ and its order,
if $T$ is represented as a subtree of $\omega^{<\omega_1}$ ordered by extension,
then the definition of $\Delta_E$ becomes independent of $T$.
For this reason, we will suppress the dependence on $T$.

If $S$ and $T$ are A-trees contained in $\omega^{<\omega_1}$ and $E \subseteq \omega_1$ is club,
define $Q_{S,T,E}$ to be the set of all finite partial functions $q:S \restriction E \to T \restriction E$
which preserve $\Delta_E$.
For reasons noted above, elements of $Q_{S,T,E}$ preserve levels
and satisfy that $s \leq t$ if and only if $q(s) \leq q(t)$.
In particular they are injective.

\begin{prop} \label{density}
    Suppose that $S$ and $T$ are pruned A-trees and $E \subseteq \omega_1$ is club such that every element of $S \restriction E$ has
    infinitely many immediate successors in $S \restriction E$ and similarly for $T \restriction E$.
    If $S$ is pruned, then $\{q \in Q : s \in \dom(q)\}$ is dense in $Q_{S,T,E}$ whenever $s \in S \restriction E$.
    If $T$ is pruned, then $\{q \in Q : t \in \ran(q)\}$ is dense in $Q_{S,T,E}$ whenever $t \in T \restriction E$.
\end{prop}

\begin{proof}
    We'll begin with the following claim.

    \begin{claim} \label{restriction_claim}
    If $q \in Q_{S,T,E}$, $t \in \dom(q)$, and $\delta \in E$ with $\delta < \Ht(t)$,
    then $q \cup \{(t \restriction \delta,q(t) \restriction \delta)\}$ is in $Q_{S,T,E}$.
\end{claim}

\begin{proof}
    We must show that if $t' \in \dom(q)$, then $\Delta_E(t',t\restriction \delta) = \Delta_E(q(t'),q(t) \restriction \delta)$.
    Notice that in general $\Delta_E(u,v \restriction \delta) = \min(\Delta_E(u,v),\delta)$.
    Thus
    $$\Delta_E(t',t \restriction \delta) = \min(\Delta_E(t',t),\delta) = \min(\Delta_E(q(t'),q(t)),\delta) = 
    \Delta_E(q(t'),q(t) \restriction \delta).$$
\end{proof}

    Let $p \in Q_{S,T,E}$ and $s \in S \restriction E$ be given.
    Let $\delta = \max \{\Delta(s',s) : s' \in \dom(q)\}$.
    By applying Claim \ref{restriction_claim} and extending $q$ if necessary,
    we may assume that $s \restriction \delta$ is in $\dom(q)$.
    If $s = s \restriction \delta$, we are done.
    Otherwise let $\delta^+$ be the least element of $E$ greater than $\delta$,
    noting that $\Ht(s) \geq \delta^+$.
    By assumption, there are infinitely many elements of $T_{\delta^+}$
    which extend $q(s \restriction \delta)$.    
    Let $t \in T$ have the same height as $s$, extend $q(s \restriction \delta)$ and
    be incompatible with every element of the range of $q$ having height greater than $\delta$.
    Thus $\Delta_E(s',s) \leq \delta$ whenever $s'$ is in the range of $q$.

    It suffices to show that $q \cup \{(t,s)\}$ is in $Q_{S,T,E}$.
    To see this suppose that $s' \in \dom(q)$, noting that $\Delta_E(s,s') \leq \delta$.
    Thus
 \begin{align*}
   \Delta_E(s,s') = \Delta_E(s \restriction \delta,s')  = \Delta_E(q(s \restriction \delta),q(s'))
    =\Delta_E(t \restriction \delta,q(s'))
    &= \min(\Delta_E(t,q(s')),\delta) \\ &= \Delta_E(t,q(s')).
\end{align*}
    The last equality holds since $q(s')$ is in the range of $q$ and hence
    $\Delta_E(t,q(s')) \leq \delta$.
    The remaining density assertion holds by reversing the roles of $S$ and $T$.
\end{proof}

If $U$ is a subtree of $T_0 \otimes \ldots \otimes T_{n-1}$
define $f_U:\omega_1 \to \omega_1$ by
$f_U(\alpha):= \beta$ if $\beta$ is minimal such that for some $u,v \in U$ with for all $i < n$, $u(i)$ is incomparable with
$v(i)$ and $\Delta(u(i),v(i)) \geq \alpha$, then
for all $i < n$, $\Delta(u(i),v(i)) \leq \beta$.
If $S$ and $T$ are A-trees, let $\Fcal_{S,T}$ be all functions of the form $f_U$ such that
$U$ is a subtree of $(S \otimes T)^{\otimes n}$, viewed as a product of $2n$ A-trees.

If $E \subseteq \omega_1$ is club and $\alpha < \omega_1$, it will be convenient to let $\next(E,\alpha)$ be the least element of $E$ greater than $\alpha$.
We will say that a club $E$ is \emph{sufficiently fast} with respect to A-trees $S$ and $T$ if for every
$f \in \Fcal_{S,T}$ there is an $\alpha \geq \min(E)$ such that
$f(\alpha) < \next (E,\alpha)$.
Equivalently, $E$ does not contain a tail of the club of $f$-closed ordinals for any $f \in \Fscr_{S,T}$.
Observe that if $U$ is a subtree of $(S \otimes T)^{\otimes n}$ and $V$ is a subtree of $U$,
then $f_U \leq f_V$.
In particular, if $\Vscr$ is a collection of subtrees of finite powers of $S \otimes T$ and every subtree of
a finite power of $S \otimes T$ contains an element of $\Vscr$,
then $\{f_V : V \in \Vscr\}$ is cofinal in $\Fscr_{S,T}$.
It follows that if $E$ is fast enough for $S$ and $T$, then this remains true in any subtree-preserving forcing extension.

\begin{rem}
    A similar definition (``\emph{fast enough club}'') appeared in \cite{univ_A-line}.
\end{rem}

\begin{thm} \label{QSTE_ccc}
    Suppose that $S$ and $T$ are A-trees with no Souslin subtrees.
    If $E$ is a club which is sufficiently fast with respect to $S$ and $T$,
    then $Q_{S,T,E}$ is c.c.c..
\end{thm}

\begin{proof}
    Let $S$, $T$, and $E$ be as in the statement of the theorem.
    Suppose that $\{q_\xi : \xi \in \omega_1\}$ is an uncountable antichain of $Q_{S,T,E}$.
    For each $\xi \in E$, let $p_\xi$ be the element of $Q_{S,T,E}$ with domain
    $$\{s \restriction \xi : s \in \dom(q_\xi) \mand \Ht(s) \geq \xi\}$$
    defined by $p_\xi(s \restriction \xi) = q_\xi(s) \restriction \xi$.
    By applying the pressing down lemma, the pigeonhole principle, and using our assumption
    that $S$ and $T$ are have no Souslin subtrees,
    there is an uncountable $\Xi \subseteq E$ and
    $\zeta \in E$ such that:
    \begin{enumerate}
        \item \label{root} $\xi \mapsto q_\xi \restriction
    \{s \in \dom(q_\xi) : \Ht(s) < \xi\}$ is constant;

        \item \label{fixed_stem} $\xi \mapsto \{(s \restriction \xi,t\restriction \xi) : (s,t) \in p_\xi\}$ is constant
        and the constant value has the same cardinality as the common value of $|p_\xi|$;

        \item \label{dom_ran_antichain}
        if $\xi \ne \eta$ are in $\Xi$,
        every element of $\dom(p_\xi)$ is incompatible with every element of
        $\dom(p_\eta)$ and every element of $\ran(p_\xi)$ is incompatible with
        every element of $\ran(p_\eta)$.
    
    \end{enumerate}
    In particular, $\dom(p_\xi)$ has a cardinality $n$ which does not depend on $\xi$ and $(s,t) \mapsto (s \restriction \zeta,t \restriction \zeta)$ is one-to-one on $q_\xi$ for all $\xi$.
    
    \begin{claim} 
    $\{p_\xi : \xi \in \Xi\}$ is an antichain.
    \end{claim}
    
    \begin{proof}
    Suppose that $\xi \ne \eta$ and
    let $s \in \dom(q_\xi)$ and $s'\in\dom(q_\eta)$ be such
    that $\Delta_E(s,s') \ne \Delta_E(q_\xi(s),q_\eta(t))$.
    Note that $s \in \dom(q_\xi) \setminus \dom(q_\eta)$ and
    $s' \in \dom(q_\eta) \setminus \dom(q_\xi)$.
    
    If $s \restriction \zeta \ne s' \restriction \zeta$,
    then there would be a $s'' \in \dom(q_\xi)$ with
    $$
    (s'' \restriction \zeta,q_\xi(s'') \restriction \zeta) =
    (s' \restriction \zeta,q_\eta(x') \restriction \zeta).$$
    Then $$\Delta_E(s,s') = \Delta_E(s,s'') = \Delta_E(q_\xi(s),q_\xi(s'')) < \zeta$$
    and therefore $\Delta_E(q_\xi(s),q_\xi(s'')) = \Delta_E(q_\xi(s),q_\eta(s''))$.
    Since this implies $$\Delta_E(s,s') = \Delta_E(q_\xi(s),q_\xi(s')),$$
    this is constriction.
    
    Thus it must be that $s \restriction \zeta = s' \restriction \zeta$.
    Since $s$ is incomparable with $s'$,
    \end{proof}

    Let's define an uncountable antichain $\{p_\xi : \xi \in \Xi\}$ satisfying $\dom(p_\xi) \subseteq S_\xi$ for 
    all $\xi \in \Xi$ as well as (\ref{fixed_stem}) and (\ref{dom_ran_antichain}) as \emph{simple};
    we have shown that if $Q_{S,T,E}$ has an uncountable antichain, it has an uncountable simple antichain.
    
    Let $U \subseteq (S \otimes T)^{\otimes n}$ be the downward closure of 
    $\{p_\xi : \xi \in \Xi\}$, where each $p_\xi$ is viewed as a tuple listing its graph with the domain
    coordinates listed with respect to a fixed lexicographic order on $S$.
    It suffices to show that there is no $\alpha \geq \min (E)$ such that
    $f(\alpha)$ is less than the least element of $E$ greater than $\alpha$.
    Toward this end, let $\alpha$ be given and let $\delta = \sup(E \cap \alpha)$.
    Let $u,v \in U$ be such that for all $i < n$, $u(i)$ is incompatible with $v(i)$,
    $\Delta(u(i),v(i)) \geq \alpha$ for all $i< n$ and $f(\alpha) = \max_{i < n} \Delta(u(i),v(i))$.
    Let $\xi,\eta \in \Xi$ be such that $p_\xi$ extends $u$ and $p_\eta$ extends $v$.
    Since $p_\xi \cup p_\eta$ is not in $Q_{S,T,E}$, there is a $\gamma \in E$ and $s \in \dom(p_\xi)$ and $s' \in \dom(p_\eta)$
    such that $\Delta(s,s') < \gamma$ iff $\Delta(p_\xi(s),p_\eta(s')) > \gamma$.
    Note that by (\ref{fixed_stem}) there is an $i < n$ such that
    $s$ extends $u(i)$ and $s'$ extends $v(i)$.
    Similarly, for some $j < n$, $p_\xi(s)$ extends $u(j)$ and $p_\eta(s')$ extends $v(j)$.
    Thus we have that
    $$\alpha \leq \min_{k < n} \Delta(u(k),v(k)) < \gamma < \max_{k < n} \Delta(u(k),v(k)) = f_U(\alpha).$$
    In particular, $\next(E,\alpha) < f_U(\alpha)$.
\end{proof}

\begin{cor} \label{coherent_club-min} \cite[Lemma 1.6]{l-maps}
    Assume $\MA_{\aleph_1}$.
    If $T$ is a coherent A-tree with no Souslin subtrees and $S$ is a subtree of $T$,
    then $T$ club-embedds into $S$.
    If $S$ and $T$ are both pruned, then they are club-isomorphic.
\end{cor}

\begin{cor}
    Let $S$ and $T$ be A-trees.
    If every finite power of $S \otimes T$ has a tree-base of cardinality $\aleph_1$,
    then there is a club $E$ such that $Q_{S,T,E}$ is c.c.c..
\end{cor}

\begin{cor} \label{fast_enough_club}
    Suppose that $P$ preserves subtrees of A-trees.
    If $\dot E$ is a $P$-name for a club subset of $\omega_1$ which does not contain a ground model
    club, then for any A-trees $S$ and $T$, $Q_{\check S,\check T, \dot E}$ is forced by $P$ to be c.c.c..
\end{cor}

\begin{cor}[essentially \cite{AbrahamShelah}] \label{baum_fast_enough}
    If $S$ and $T$ are A-trees in a model $V$ of ZFC and $E \subseteq \omega_1$ a $V$-generic Baumgartner club, then
    in $V[E]$, $Q_{S,T,E}$ is c.c.c..
\end{cor}

\begin{rem}
    The poset to force two A-trees to be club isomorphic defined in \cite{AbrahamShelah} is equivalent to the poset of all $q=(E_q,f_q)$ such that $E_q$ is a finite subset of $\omega_1$ and $f_q$ is in $Q_{S,T,E_q}$.
    This is isomorphic to a suborder of $P * Q_{\check S,\check T,\dot E}$ where $P$ is the poset to add a Baumgartner club.
    The Abraham-Shelah proof of properness works \emph{mutatis mutandis} to establish that $P * Q_{\check S,\check T, \dot E}$
    is proper. 
\end{rem}

\begin{cor}
    $\diamondsuit$ is consistent with the statement that for all pruned A-trees $S$ and $T$ there is a c.c.c.
    poset which forces $S$ is club-isomorphic to $T$.
\end{cor}

\begin{proof}
First observe that the poset $\Add(\aleph_1,1)$ is $\sigma$-closed and adds a club $\dot E$ which does not contain a ground model club.
Since it is $\sigma$-closed, it preserves A-trees \cite[Lemma 3.1]{baumbase} \cite[Lemma 8.13]{trees_lo:Todorcevic} and thus by Corollary \ref{fast_enough_club},
if $S$ and $T$ are A-trees, $\Add(\aleph_1,1)$ forces $Q_{S,T,E}$ is c.c.c..
Now if $V$ satisfies CH and $G \subseteq \Add(\aleph_1,\aleph_2)$ is $V$-generic, then for any A-trees $S$ and $T$ in $V[G]$,
there is a $V$-generic filter $G_0 \subseteq \Add(\aleph_1,\aleph_2)$ and a $V[G_0]$-generic filter
$G_1 \subseteq \Add(\aleph_1,1)$ such that
$V[G] = V[G_0][G_1]$ and $S,T \in V[G_0]$.
The corollary now follows by applying our observation in $V[G_0]$ to the generic extension $V[G_0][G_1] = V[G]$.
\end{proof}

\section{Independence results with \texorpdfstring{$\MA_{\aleph_1}$}{MA(w1)}}
\label{indep_MA:sec}

In this section we prove the main independence results of the paper.
The next proposition is an easy consequence of results in \cite{kunen_isbell} and will serve as the basis establishing
$\MA_{\aleph_1}$ does not suffice for certain applications of PFA to the theory of Aronszajn trees.

\begin{prop} \label{many_U(T)_extend_club}
If $\ZFC$ is consistent, then so is $\ZFC + \MA_{\aleph_1}$ together with there is a family $\Sscr$ of coherent
A-trees such that:
\begin{enumerate}
    \item $|\Sscr| = 2^{\aleph_1}$;
    \item if $S \in \Sscr$, then $\Uscr(S)$ extends the club filter;
    \item if $S \ne S'$ are in $\Sscr$, then $\Uscr(S) \ne \Uscr(S')$.
\end{enumerate}
\end{prop}

\begin{proof}
    In \cite{kunen_isbell} it was proved that if $\Fscr$ is
    a uniform filter on $\omega_1$, then
    there is a c.c.c. poset which forces the existence of coherent A-tree $S$ such
    that $\Uscr(S)$ extends $\Fscr$.
    Start with a model of ZFC in which there is a family $\Ascr$ of stationary subsets of
    $\omega_1$ such that $|\Ascr| = 2^{\aleph_1}$ and any two elements of $\Ascr$ have nonstationary intersection
    (this is true for instance in any model of $\diamondsuit$).
    Construct an finite support iteration of c.c.c. forcings of length $2^{\aleph_1}$ such that at stage $\xi$,
    a coherent A-tree $S_\xi$ is added by a c.c.c. forcing of cardinality at most $2^{\aleph_1}$ such that
    $\Uscr(S_\xi)$ extends the club filter and contains the $\xi\Th$-element of $\Ascr$.
    (Here we utilize the well-know fact that
    every club subset of $\omega_1$ in a c.c.c. generic extension contains
    a groundmodel club.)
    Finally, force $\MA_{\aleph_1}$ without enlarging $2^{\aleph_1}$
    noting that in this final model each $\Uscr(S)$ for $S \in \Sscr$ is an ultrafilter
    extending the club filter and $S \mapsto \Uscr(S)$ is one-to-one.
\end{proof}

\begin{thm}
It is relatively consistent with ZFC that $2^{\aleph_0} < 2^{\aleph_1}$
and that for any two pruned A-trees $S$ and $T$, 
there is a c.c.c. forcing extension in which $S$ is club-isomorphic to $T$.
In particular, it is relatively consistent with ZFC that there are two
non-club isomorphic A-trees which can be made club isomorphic via a c.c.c. forcing.
\end{thm}
\begin{proof}
Start with a model of $2^{\aleph_0} = \aleph_1$ and $2^{\aleph_1} = \aleph_3$
and iterate Baumgarter's club forcing with length $\omega_2$ using countable
support.
The iteration $P_{\omega_2}$ preserves $\aleph_1$ and $\aleph_2$
(since it is proper and $\aleph_2$-c.c.).
By a standard name-counting argument,
the iteration forces $2^{\aleph_0} = \aleph_2 < 2^{\aleph_1} = \aleph_3$.
Now suppose that $\dot S$ and $\dot T$ are $P_{\omega_2}$-names
for A-trees.
Since $P_{\omega_2}$ is $\aleph_2$-c.c.c., there is a $\delta < \omega_2$
such that $\dot S$ and $\dot T$ are $P_\delta$-names.
Let $\dot E$ be the $P_{\delta+1}$-name for the $\delta\Th$ generic club. 
By the observations made before Theorem \ref{QSTE_ccc}, $\dot E$ is forced to be sufficiently fast for
$\dot S$ and $\dot T$.
Since the iterands are strongly proper, they preserve subtrees
of $\omega_1$-trees by Theorem \ref{strong_proper_pres_subtree}.
By \cite[5.1]{con_linear_basis2}, the remainder of
the iteration $P_{\omega_2}/\dot G_\delta$ preserves subtrees of $\omega_1$ trees and hence $\dot E$ is forced
by $P_{\omega_2}$ to be fast enough for $\dot S$ and $\dot T$.
By Theorem \ref{QSTE_ccc} applied in the generic extension, $P_{\omega_2}$ forces that $Q_{\dot S, \dot T, \dot E}$ is c.c.c..
By Proposition \ref{density}, $P_{\omega_2} * Q_{\dot S,\dot T, \dot E}$ forces $\dot S$ and $\dot T$ are club isomorphic.
\end{proof}

\begin{thm} \label{MAomega1+coherent_A-trees_noniso}
    Assuming ZFC is consistent, there is a model of $\ZFC+\MA_{\aleph_1}$ in which the following statements are true:
    \begin{enumerate}
        \item  \label{many_non_near_A-trees}
        There is a family $\Sscr$ of $2^{\aleph_1}$ coherent A-trees such that no pair of elements of $\Sscr$ have
        club-isomorphic subtrees.

        \item \label{many_non_near_C-lines}
        There is a family $\Cscr$ of $2^{\aleph_1}$ Countryman lines such that no pair of elements of $\Cscr$ have 
        uncountable suborders which are isomorphic or reverse isomorphic.
        
        \item There is an A-tree which is not club isomorphic to a coherent A-tree.

        \item \label{con3} There are two pruned A-trees $S$ and $T$ and a strictly increasing level preserving surjection
    $f:S \to T$ such that $S$ and $T$ are not club isomorphic.

        \item \label{con4} There are two pruned A-trees $S$ and $T$ which are not club
    isomorphic and yet $S \leq T$ and $T \leq S$.
    Moreover $S$ can be taken to be coherent.
    
    \end{enumerate}
\end{thm}

\begin{rem}
Item (\ref{many_non_near_C-lines}) contradicts both \cite[Theorem 1.1]{univ_A-line} and \cite[Corollary 2.1.13]{walks}
(in the former case, note that any A-line contains an $\aleph_1$-dense nonstationary A-line).
Theorem 1.1 of \cite{univ_A-line} referenced \cite{walks}; the proofs in \cite{univ_A-line} are correct but establish the weaker
result that $\MA_{\aleph_1}$ implies that every Countryman line is minimal,
a result first (essentially) proved in \cite[Theorem 2.1.12]{walks}.
The issue with the proof of \cite[Corollary 2.1.13]{walks} can be summarized more generally as follows:
while any set $\{q_\xi : \xi \in \omega_1\}$ of conditions in a poset such as $Q_{S,T,E}$ can be refined to
some $\{q_\xi : \xi \in \Xi\}$ such that $\{\Delta(\dom(q_\xi)) : \xi \in \Xi\}$ and $\{\Delta(\ran(q_\xi)) : \xi \in \Xi\}$
form $\Delta$-systems, the roots of these $\Delta$-systems can have different cardinalities.
\end{rem}
\begin{proof}
    By Proposition \ref{many_U(T)_extend_club} there is a model of $\MA_{\aleph_1}$ in which there is a family $\Sscr$ consisting of $2^{\aleph_1}$
    coherent A-trees for which the ultrafilters $\Uscr(S)$ are all distinct.
    By replacing these trees with subtrees if necessary, we may assume that they are all pruned.
    By Corollary \ref{coherent_club-min}, it suffices to show that if $U$ and $V$ are distinct elements of $\Sscr$,
    then $U$ and $V$ are not club isomorphic.
    Suppose for contradiction that $E \subseteq \omega_1$ is club and 
    $f:U_E \to V_E$ is an isomorphism.
    Let $X \subseteq E$ be such that $X \in \Uscr(U) \setminus \Uscr(V)$.
    Let $A \subseteq U_E$ be an uncountable antichain such
    that $\Delta(A) \subseteq X$.
    Since $\Delta(A) \subseteq E$ and since $f$ is an isomorphism,
    $\Delta(f(A)) = \Delta(A) \subseteq X$.
    This contradicts that $X \not \in \Uscr(V)$.

    Let $\Cscr$ consist of the trees in $\Sscr$ equipped with their lexicographic ordering.
    Noting that $\MA_{\aleph_1}$ implies all A-trees are special, \cite[]{walks} implies that $\Cscr$ consists of Countryman lines.
    By the properties of $\Sscr$ and by Corollary \ref{coherent_club-min},
    no pair of elements of $\Cscr$ contain uncountable isomorphic or reverse isomorphic suborders.
    
    Next observe that since $U$ and $V$ are both club-isomorphic to subtrees of $U \oplus V$,
    $U \oplus V$ is not club-isomorphic to a coherent tree:
    otherwise Theorem \ref{coherent_club-min} would imply
    that $U \oplus V$ is club-isomorphic to both $U$ and $V$ and hence that $U$ and $V$ would be club isomorphic,
    contrary to our assumption.

Next observe that that $U \otimes V$ is coherent and is club-isomorphic to at most one of $U$ 
or $V$.
For concreteness, suppose that $U$ is not club isomorphic to $U \otimes V$.
Since the projection maps from $U \otimes V$ to $U$ and to $V$
are strictly increasing and order preserving, we have established (\ref{con3}) with $S = U \otimes V$ and $T = U$.

Finally we claim that (\ref{con4}) is witnessed by $S = U$ and $T = U \oplus U \otimes V \cong U \otimes (\omega_1 \oplus V)$.
Clearly $S \leq S + U \otimes V$ and the 2nd coordinate projection
witnesses $T \cong U \otimes (\omega_1 \oplus V) \leq U = S$.
If $S$ and $T$ were club isomorphic, then 
since $U \otimes V$ is club isomorphic to a pruned subtree of $T$ and hence
is club-isomorphic to a pruned subtree of $S$, Theorem \ref{coherent_club-min} implies
that $S = U$ is club-isomorphic to $U \otimes V$, contrary to our assumption.
\end{proof}

\section{Two consequences of \texorpdfstring{$\MA_{\aleph_1}$}{MAomega1} concerning coherent trees}
\label{two_consequences:sec}
We conclude this paper  with two additional consequences of $\MA_{\aleph_1}$.
The first shows that (\ref{con4}) of Theorem \ref{MAomega1+coherent_A-trees_noniso} is sharp.

\begin{prop}
    Assume $\MA_{\aleph_1}$ and suppose $S$ and $T$ are pruned A-trees and $S$ is coherent.
    If $S \leq T$ and $T \leq S$, then $S$ is club isomorphic to a subtree of $T$.
    If moreover $T$ is coherent, then $S$ is club isomorphic to $T$.
\end{prop}

\begin{proof}
    Let $f:S \to T$ and $g:T \to S$ be strictly increasing level preserving maps and let
    $h(s) = g(f(s))$.
    By \cite[Lemma 1.10]{l-maps}, there is an uncountable antichain $A \subseteq S$ such that for all
    $s \ne s'$ in $A$, $\Delta(h(s),h(s')) = \Delta(s,s')$.
    Since 
    $$\Delta(s,s') \leq \Delta(f(s),f(s')) \leq \Delta(g(f(s)),g(f(s'))) = \Delta(h(s),h(s')),$$
    it follows that $\Delta(s,s') = \Delta(f(s),f(s'))$ for all $s \ne s'$ in $A$.
    Let $U$ consist of all $u \in S$ such that $u$ has uncountably many extensions in $A$.
    Note that $U$ is pruned: if $u \in U$ and $\alpha \geq \Ht(u)$, then since $S_\alpha$ is countable
    the pigeonhole principle implies some $u' \in S_\alpha$ extending $u$ has uncountably many extensions in $A$.

    We next claim that $f \restriction U$ is an isomorphism.
    Since $f$ is level preserving and strictly increasing, it suffices to show that $f$ is one-to-one on levels.
    Toward this end, suppose that $u \ne u'$ are in $U$ and have the same height $\alpha$.
    Let $s$ and $s'$ be elements of $A$ extending $u$ and $u'$ respectively.
    Then $\Delta(s,s') = \Delta(u,u') = \Delta(f(s),f(s')) < \alpha$.
    Since $f(u) \leq f(s)$, $f(u') \leq f(s')$, and $\Ht(f(u)) = \Ht(f(u')) = \alpha$,
    it follows that $f(u) \ne f(u')$.

    Now, by Corollary \ref{coherent_club-min}, $S$ is club-isomorphic to $U \cong f(U)$.
    Moreover, if $T$ is coherent as well, $T$ is club-isomorphic to $f(U)$.
\end{proof}

The context for the next result is that by \cite[Corollary 6.8]{l-maps}, PFA implies that there is a unique ultrafilter, up to Ruden-Keisler equivalence
of the form $\Uscr(T)$ for $T$ a coherent A-tree.
Moreover by \cite{kunen_isbell}, PFA implies that $\Uscr(T)$ is minimal in the Ruden-Keisler order with respect
to being a uniform ultrafilter on $\omega_1$ and by \cite{UT_selective} it is an immediate $\leq_{\mathrm{RK}}$-successor
of a selective ultrafilter on $\omega$.

\begin{thm}
    Assume $\MA_{\aleph_1}$.
    If $S$ and $T$ are club-isomorphic A-trees and $S$ is coherent, then
    there is a countable-to-one function $f:\omega_1 \to \omega_1$ such that $\Uscr(S) = f^* \Uscr(T)$.
\end{thm}

\begin{proof}
    Let $E \subseteq \omega_1$ be a club
    and $\phi:T \restriction E \to S \restriction E$ be an isomorphism.
    Let $Q$ be the poset of all finite antichains $q \subseteq S \restriction E$ such that
    $\Delta(q)$ is separated by $E$.
    It is easily checked that $Q$ is c.c.c. and therefore by an application of $\MA_{\aleph_1}$ to $Q$, there
    is an uncountable antichain $A \subseteq S \restriction E$ such that $\Delta(A)$ is separated by $E$.
    Define $f:\omega_1 \to \Delta(A)$ so that $f(\alpha) = \delta$ if $\sup(E \cap \alpha) = \sup ( E \cap \delta)$.
    Since $E$ and $\Delta(A)$ are unbounded, $f$ is countable-to-one.

    To see that $f^*\Uscr(T) = \Uscr(S)$, suppose that $X \in \Uscr(T)$ and let $B \subseteq T$
    be an uncountable antichain such that $\Delta(B) \subseteq X$.
    Select $t_\xi \in T_\xi$ for $\xi \in E$ such that $t_\xi$ is extended by an element of $B$ and let
    $s_\xi \in S_\xi$ be extended by an element of $A$.
    By \cite[Lemma 1.10]{l-maps}, there is an uncountable $\Xi \subseteq \omega_1$ such that
    for all $\xi \ne \eta$ in $\Xi$, $s_\xi$ is incompatible with $s_\eta$ and 
    $\Delta(s_\xi,s_\eta) = \Delta(\phi(s_\xi),\phi(s_\eta))$.
    Since $\phi$ is an isomorphism, $\Delta(t_\xi,t_\eta)$ and $\Delta(\phi(t_\xi),\phi(t_\eta)) = \Delta(s_\xi,s_\eta)$
    are not separated by $E$ and thus $f(\Delta(t_\xi,t_\eta)) = \Delta(s_\xi,s_\eta)$.
    It follows that $$f(\Delta(\{t_\xi : \xi \in \Xi\}))=\Delta(\{s_\xi : \xi \in \Xi\}) \in \Uscr(S)$$
    and since $\Delta(\{t_\xi : \xi \in \Xi\})\subseteq \Delta(B)\subseteq X,$ we get that $f(X)\in \Uscr(S).$
\end{proof}

\end{document}